\documentclass[12pt,letterpaper,reqno]{amsart}

\usepackage{amssymb}
\usepackage[hidelinks]{hyperref}
\usepackage[T1]{fontenc}

\usepackage{txfonts}  

\newcommand{\N}{\mathbb N}
\newcommand{\Z}{\mathbb Z}
\newcommand{\R}{\mathbb R}
\newcommand{\C}{\mathbb C}
\newcommand{\T}{\mathbb T}
\newcommand{\id}{\textup{id}}

\begin{document}

\title{Extensions of the Laurent Decomposition and the spaces $A^p(\Omega)$}
\author{Nikolaos Georgakopoulos}
\address{University of Athens, Department of Mathematics, 157 84 Panepistemiopolis, Athens, Greece}
\email{nikolaosgmath@gmail.com}

\theoremstyle{plain}

\newtheorem{prop}{Proposition}
\newtheorem{thm}[prop]{Theorem}
\newtheorem{cor}[prop]{Corollary}
\newtheorem{lem}[prop]{Lemma}

\theoremstyle{definition}
\newtheorem{defn}[prop]{Definition}
\newtheorem{rem}[prop]{Remark}
\newtheorem{que}[prop]{Question}
\newtheorem{nota}[prop]{Notation}

\begin{abstract} 
We generalize the classical Laurent decomposition in the setting of domains $\Omega\subseteq \C$ bounded by Jordan curves. This leads us to study the Fr\'echet spaces $A^p(\Omega)$, and their relation to the spaces $C^p(\partial \Omega)$. In the final section, we examine the case of a non Jordan domain $\Omega$. \end{abstract}

\maketitle

\tableofcontents

\section{Introduction}

\indent It is well known that if $\Omega\subseteq \C$ is a domain of finite connectivity, then every $f\in H(\Omega)$ has a decomposition on $\Omega$ as $f=f_0+f_1+\cdots+f_n$ for $f_k\in H(V_k)$, where $V_1,...,V_n$ denote the complements of the connected components of $(\C\cup \{\infty\})\setminus \Omega$. This is the Laurent decomposition of $f$; see for instance \cite{Costakis}. 
If $f$ extends continuously over $\overline{\Omega}$, then it is easily seen that each $f_k$  extends continuously over $\overline{V_k}$. More generally, if $f\in A^p(\Omega)$ then $f_k\in A^p(V_k)$, where for $p\in \{0,1,2,...\}\cup \{\infty\}$, $A^p(\Omega)$ denotes the set of functions $f\in H(\Omega)$ whose derivatives $f^{(l)}$, $0\le l\le p$, $l\in \N$,  have continuous extensions over $\overline{\Omega}$.

As a special case, we can take $\Omega$ to be a domain bounded by a finite number of pairwise disjoint Jordan curves, such as the annulus $\Omega=\{z\in \C:r<|z|<1\}$, $0<r<1$. The limit case as $r\to 1^-$ is the case of the unit circle $\T$. The Laurent decomposition can be generalized in this setting, as every function $f\in C^{\infty}(\T)$ can be written as $f=g+h$ with $g\in A^{\infty}(D)$ and $h\in A^{\infty}(\C\setminus \overline{D})$, $\lim_{z\to \infty}h(z)=0$, $D$ being the open unit disc. The analogous statement for the spaces $C^p(\T),A^p(D)$ and $A^p(\C\setminus \overline{D})$ does not hold for $p<+\infty$; this is related to the fact that the disc algebra $A(D)=A^0(D)$ is not complemented in $C(\T)=C^0(\T)$ (\cite{Rudin}). We prove this result in section \ref{CircleCaseSection}, along with the fact that the spaces $A^p(D)$ are isomorphic Banach spaces for $p<+\infty$ (with their natural norms) and the fact that the space $A^{\infty}(D)$ is a non normable Fr\'echet space.

In section \ref{JordanDomainSection}, we examine potential generalizations of the previous results, when the disc $D$ is replaced by a Jordan domain $\Omega$. The spaces $C^p(\partial \Omega)$, $0\le p\le +\infty$, are defined via the parametrization of $\partial \Omega$ induced by any Riemann mapping of $D$ onto $\Omega$. In order to extend the results of the preceding section, we place additional hypotheses on the geometry of $\Omega$, such as the interior chord arc condition (\cite{Chord Arc}) and the boundedness of the geodesics (\cite{HInfinity}), or on the Riemann mapping $\phi:D\to \Omega$ itself, such as $\phi\in A^p(D)$ and $\phi'(z)\neq 0$ for all $z\in \overline D$. Under some of those conditions, we also prove that the spaces $A^p(\Omega)$ and $A^p(D)$ are isomorphic as Banach spaces, for $p<+\infty$, and that $A^{\infty}(\Omega)$ has no norm inducing its natural topology.

In section \ref{InternallTangentSection}, we consider the case of a particular domain $\Omega$ bounded by two Jordan curves meeting at a point. For $p=+\infty$ the ''Laurent decomposition'' is true, but the case $p<+\infty$ remains open. We suspect the answer is negative, and we have reduced this to the existence of a function $f\in A(\Omega)$ so that the Cauchy transform of $f|_{\T}$ ($\T\subseteq \partial \Omega$) diverges as $z\to 1$, $|z|<1$ (see Question \ref{FinalQuestion}).

Section \ref{PrelimSection} contains some preliminary results needed in the sequel.

Finally, we remark that if $\Omega$ is a domain of finite connectivity, then $A^p(\Omega)$ is a finite direct sum of spaces $A^p(V_k)$, so we can extend our results in this setting, under some additional assumptions on $\Omega$. We do not discuss such extensions in the present article, nor do we discuss the several variables case or the density of polynomials or rational functions in $A^p(\Omega)$. These facts will be treated elsewhere.

\section{Preliminaries}\label{PrelimSection}

We first collect two elementary facts of functional analysis, that we will use throughout this article. The following is well known

\begin{prop}\label{Functional1}Let $B$ be a topological vector space and $A$ be a subspace of $B$. The following are equivalent
\begin{itemize}\item[1.] There is a continuous linear map $B\to A$ that fixes $A$ (a projection $B\to A$).
\item[2.] There is another subspace $C$ of $B$ so that $B=A+C$, $A\cap C=0$ and the projections $B\to A$ and $B\to C$ are continuous.\end{itemize}\end{prop}

If the second item holds, $B$ is the direct sum of $A$ and $C$, $B=A\oplus C$, and $A$ is complemented in $B$. Note that this forces $A,C$ to be closed subspaces of $B$.

\begin{prop}\label{Functional2}Let $B$ be a Fr\'echet space and $A,C$ be subspaces of $B$ with $A+C=B$, $A\cap C=0$. The following are equivalent
\begin{itemize}\item[1.] $B=A\oplus C$.
\item[2.] $A$ and $C$ are closed in $B$.
\item[3.] $A$ and $C$ are Fr\'echet spaces in the induced topology.
\end{itemize}\end{prop}

\begin{proof}We will prove that condition 3 implies condition 1. If $A,C$ are Fr\'echet spaces, $A\times C$ is a Fr\'echet space (with the sum of the semi-norms) and the canonical map $A\times C\to B$ is a continuous linear bijection. By the Open Mapping Theorem it is an isomorphism, so the projection $B\to A\times C\to A$ is continuous.
\end{proof}

For $0\le p\le +\infty$, we denote by $A^p(D)$ the space of holomorphic functions on $D$ whose derivatives of order $l\in \N$, $0\le l\le p$,  extend continuously over $\overline D$. It is topologized via the semi-norms 
\begin{equation*}|f|_l=\sup_{z\in \overline D}|f^{(l)}(z)|=\sup_{z\in \T}|f^{(l)}(z)|\text{ , }0\le l\le p\text{ , }l\in \N\end{equation*}
$A^p(D)$ is a Banach space for $p<+\infty$ and a Fr\'echet space for $p=+\infty$. The disc algebra is $A(D)=A^0(D)$.

 $A_0^p(\hat \C\setminus \overline D)$ is the space of functions $f$ that are holomorphic on $D^c$, whose derivatives of order $l\in \N$, $0\le l\le p$, extend continuously over $\C\setminus D$, and $\lim_{z\to \infty}f(z)=0$. Its topology is given by the semi-norms
\begin{equation*}|f|_l=\sup_{z\in D^c}|f^{(l)}(z)|=\sup_{z\in \T}|f^{(l)}(z)|\text{ , }0\le l\le p\text{ , }l\in \N\end{equation*}
$A_0^p(\hat \C\setminus \overline D)$ is a Banach space for $p<+\infty$ and a Fr\'echet space for $p=+\infty$.  Here, $\hat \C$ denotes the extended plane $\C\cup \{\infty\}$.

We will need the following well known fact 

\begin{prop}If $f\in A(D)$ then $\hat f(n)=0$ for $n<0$ and if $f\in A_0(\hat{\C}\setminus \overline D)$ then $\hat f(n)=0$ for $n\ge 0$. In both cases, the Laurent coefficients are the Fourier coefficients of the $2\pi$-periodic function $f(e^{i\theta})$, $\theta\in \R$.\end{prop}

\begin{prop}\label{A^pC^pDerivatives}Let $f\in A(D)$ and $p\ge 1$. Then $f\in A^p(D)\iff f|_{\T}\in C^p(\T)$. In that case, 
\begin{equation*}\frac{d^lf}{(de^{i\theta})^l}=f^{(l)}\text{ , }1\le l\le p\text{ , }l\in \N\end{equation*}
where $f^{(l)}$ is the continuous extension over $\T$ of the complex derivative of order $l$ of $f$ on $D$.\\
The analogous statement is true for $f\in A(\hat{\C}\setminus \overline D)$.\end{prop}

\begin{proof}See \cite{Vlasis}.\end{proof}
\noindent We note that
\begin{equation*}\frac{df}{d\theta}=\frac{df}{de^{i\theta}}ie^{i\theta}\end{equation*}

\noindent For $p=0,1,....,+\infty$, the topology of $C^p(\T)$ is given by the semi-norms
\begin{equation*}\Big\|\frac{d^lf}{d\theta^l}\Big\|_{\infty}\text{ , }0\le l\le p\text{ , }l\in \N\end{equation*}
or equivalently by the semi-norms
\begin{equation*}\Big\|\frac{d^lf}{(de^{i\theta})^l}\Big\|_{\infty}\text{ , }0\le l\le p\text{ , }l\in \N\end{equation*}
The topology of $A^p(D)$ is also induced by the semi-norms
\begin{equation*}\sup_{z\in \T}|f^{(l)}(z)|=\Big\|\frac{d^lf}{(de^{i\theta})^l}\Big\|_{\infty}\end{equation*}
It follows that the restriction map $A^p(D)\to C^p(\T)$ is an embedding, and since the spaces $A^p(D), C^p(\T)$ are complete, we have:

\begin{prop}For all $p=0,1,...,+\infty$, $A^p(D)$ and $A^p_0(\hat{\C}\setminus \overline D)$ are closed subspaces of $C^p(\T)$, with trivial intersection.\end{prop}

\section{The case of the circle}\label{CircleCaseSection}

We begin by showing the analogue of the Laurent decomposition for functions in $C^{\infty}(\T)$. To do that, we first collect some well known results regarding the asymptotic behavior of the Fourier coefficients of functions in $C^{\infty}(\T),A^{\infty}(D)$ and $A^{\infty}_0(\hat{\C}\setminus \overline D)$:

\begin{itemize}\item A function $f\in C(\T)$ is in $C^{\infty}(\T)$ if and only if\, $n^l\hat f(n)=0$ as $|n|\to +\infty$ for all $l\ge 0$.
\item A function $f\in A(D)$ is in $A^{\infty}(D)$ if and only if $n^l\hat f(n)\to 0$ as $n\to +\infty$, for all $l\ge 0$
\item A function $f\in A_0(\hat{\C}\setminus \overline D)$ is in $A^{\infty}_0(\hat{\C}\setminus \overline D)$ if and only if $n^l\hat f(n)\to 0$ as $n\to -\infty$, for all $l\ge 0$
\end{itemize}

It is also easy to see that, if $f\in A^{\infty}_0(\hat{\C}\setminus \overline D)$ then $g(z)=f(1/z)$, $g(0)=0$, is in $A^{\infty}(D)$.

The usual topologies of $C^{\infty}(\T), A^{\infty}(D), A^{\infty}_0(\hat{\C}\setminus \overline D)$ are also given by these two (equivalent) families of semi-norms:
\begin{equation}\label{Semi-Norm Equivalence}\sup_{n\in \Z}|\hat f(n)|\text{ \ , \ }\sup_{n\in \Z}|n^l\hat f(n)|\text{ \ for }l>0\text{ , }l\in \N\end{equation}
\begin{equation}\label{Semi-Norm Equivalence2}|\hat f(n)|+\sum_{n\in \Z, n\neq 0}|n^l\hat f(n)|\text{\ , }l\in \N\end{equation}
If $f\in A^{\infty}(D)$, then in $\eqref{Semi-Norm Equivalence}$ and $\eqref{Semi-Norm Equivalence2}$ we can have $n$ range over $\N=\{0,1,...\}$, and if $f\in A^{\infty}_0(\hat{\C}\setminus \overline D)$ we can have $n$ range over the negative integers.

\begin{thm}\label{CInfinityDecomposition}Every $f\in C^{\infty}(\T)$ can be uniquely decomposed as $f=g+h$ for $g\in A^{\infty}(D)$ and $h\in A_0^{\infty}(\hat \C\setminus \overline D)$. Moreover,
\begin{equation*}C^{\infty}(\T)=A^{\infty}(D)\oplus A_0^{\infty}(\hat \C\setminus \overline D)\end{equation*}\end{thm}

\begin{proof}If $f\in C^{\infty}(\T)$ consider $g(z)=\sum_{n\ge 0}\hat f(n)z^n$, $|z|<1$, and $h(z)=\sum_{n<0}\hat f(n)z^n$, $|z|>1$, $h(\infty)=0$. By the discussion in the beginning of this section, $g\in A^{\infty}(D)$, $h\in A^{\infty}_0(\hat \C\setminus \overline D)$, and of course $f=g|_{\T}+h|_{\T}$. 
The decomposition is unique since $f=0\implies g=-h$ which implies that all Fourier coefficients of $g,h$ are $0$. Finally, we can use Proposition \ref{Functional2} to derive the result regarding the direct sum.\end{proof}

There is an equivalent way to state this decomposition:

\begin{thm}\label{CInfinityConjugateDisc}Every $f\in C^{\infty}(\T)$ can be uniquely decomposed as $f=g+\bar h$, $g,h\in A^{\infty}(D)$ and $h(0)=0$.\end{thm}

\begin{proof}Let $f\in C^{\infty}(\T)$. We have 
\begin{equation*}f(z)=\sum_{n=-\infty}^{+\infty}a_nz^n=\sum_{n=0}^{+\infty}a_nz^n+\sum_{n=-\infty}^{-1}a_nz^n\text{   ,  }|z|=1\end{equation*}
 The second summand converges for $|z|>1$, and when restricted to $|z|=1$ we have that
\begin{equation*}\overline{\sum_{n=-\infty}^{-1} a_n z^n}=\sum_{n=-\infty}^{-1}\bar a_n(\bar z)^n=\sum_{n=-\infty}^{-1}\bar a_nz^{-n}=\sum_{n=1}^{+\infty}\bar a_{-n}z^n=h(z)\end{equation*}
where $h\in A^{\infty}(D)$ with $h(0)=0$. So $f=g+\bar h$, with $g(z)=\sum_{n=0}^{+\infty}a_nz^n$. The decomposition is unique, for if $f=0$ then $g=-\bar h$ forcing $g$ and $h$ to be constant on $D$.
\end{proof}

Consider $\overline{A_0^{\infty}(D)}=\{\bar f:f\in A^{\infty}(D)\text{ , }f(0)=0\}$, topologized via the semi-norms
\begin{equation*}\|\big(\overline f\big)^{(l)}\|_{\infty,\T}\text{ , }l\in \N\end{equation*}
rendering the conjugation map $\overline{A_0^{\infty}(D)}\to A_0^{\infty}(D)$ an isometric isomorphism. We embed $\overline{A_0^{\infty}(D)}$ into $C^{\infty}(\T)$ via the map $f\mapsto f|_{\T}$, and under this embedding, we have
\begin{equation*}\overline{A_0^{\infty}(D)}=A_0^{\infty}(\hat \C\setminus \overline D)\end{equation*}
This was shown in the proof of Theorem \ref{CInfinityConjugateDisc}. By Theorem \ref{CInfinityDecomposition} we obtain,

\begin{cor}$C^{\infty}(\T)=A^{\infty}(D)\oplus \overline{A_0^{\infty}(D)}$.\end{cor}

In the rest of this section, we shall prove that similar decompositions are impossible for $p<+\infty$. We begin with some useful isomorphisms, the first of which was communicated to us by C. Panagiotis.

\begin{thm}\label{IsomorphismsDisc}For all $p<+\infty$, there is an isomorphism $\Phi:C^{p+1}(\T)\to C^p(\T)$ that restricts to isomorphisms $A^{p+1}(D)\to A^p(D)$ and $A^{p+1}_0(\hat{\C}\setminus \overline D)\to A^p_0(\hat{\C}\setminus \overline D)$. \\
It follows that the spaces $C^0(\T), C^1(\T), C^2(\T),...$ are isomorphic Banach spaces. The same is true for the spaces $A^0(D), A^1(D),A^2(D),...$ .\end{thm}

\begin{proof}Let $\Phi:C^{p+1}(\T)\to C^p(\T)$ be given by
\begin{equation}\label{CC^1Isomorphism}\Phi(f)=\frac{df}{d\theta}+\hat f(0)\end{equation}
This map is clearly continuous and linear, and it is easy to see that it is injective (a linear function $a\theta+b$, $a,b\in \C$, must be $2\pi$-periodic hence constant). We shall now show surjectivity. Let $g\in C^p(\T)$ and consider the function $S_g\in C^{p+1}(\T)$ given by,
\begin{equation*}S_g(e^{i\theta})=\int_0^{\theta}g(e^{i\omega})d\omega-\hat g(0)\theta\end{equation*}
If $f\in C^{p+1}(\T)$ is given by
\begin{equation}\label{CC^1Isomorphism2}f=S_g-\widehat{S_g}(0)+\hat g(0)\end{equation}
then we can easily check that $g=\Phi(f)$, so $\Phi$ is surjective. By the Open Mapping Theorem, $\Phi:C^{p+1}(\T)\to C^p(\T)$ is an isomorphism, and its inverse is given by the assignment $g\mapsto f$, $f$ being as in $\eqref{CC^1Isomorphism2}$.

We shall now prove that $\Phi$ restricts to an isomorphism $A^{p+1}(D)\to A^p(D)$. If $f\in A^{p+1}(D)$, we have
\begin{equation*}\Phi(f)(z)=\frac{df}{d\theta}(z)+\hat f(0)=\frac{df}{de^{i\theta}}(z)ie^{i\theta}+f(0)=f'(z)zi+f(0)\end{equation*}
for $z=e^{i\theta}$. This allows us to extend $\Phi(f)$ over the closed disc. Therefore, we have the map $\Phi:A^{p+1}(D)\to A^p(D)$, 
\begin{equation}\label{AA^1Isomorphism1}\Phi(f)(z)=f'(z)zi+f(0)\end{equation}
$\Phi$ is clearly a continuous linear injection $A^{p+1}(D)\to A^p(D)$. Before we show surjectivity, let us note that every function in $A^{p+1}(D)$ has an antiderivative in $A^p(D)$. Indeed, 
\begin{equation*}F(z)=\int_{[0,z]}f(\zeta)d\zeta\end{equation*}
 is an antiderivative of $f$ on $D$, and it is easy to check the uniform continuity of $F$ over $D$, hence $F$ extends continuously over $\overline D$. The first $p$ derivatives of $F$ also extend continuously over $\overline D$, so $F\in A^{p+1}(D)$. To show that $\Phi:A^{p+1}(D)\to A^p(D)$ is surjective, let $g\in A^p(D)$. The function
\begin{equation*}r(z)=\begin{cases}\dfrac{g(z)-g(0)}z&\textup{if, }z\neq 0\\
g'(0)&\textup{if, }z\neq 0\end{cases}\end{equation*}
has an antiderivative $G$ with $G(0)=0$. Consider $f:\overline D\to \C$,
\begin{equation}\label{AA^1Isomorphism2}f(z)=-iG(z)+g(0)\end{equation}
One can easily check that $f\in A^{p+1}(D)$ and $\Phi(f)=g$, so $\Phi:A^{p+1}(D)\to A^p(D)$ is also surjective. By the Open Mapping Theorem, $\Phi$ is an isomorphism, with inverse given by $g\mapsto f$, $f$ as in $\eqref{AA^1Isomorphism2}$.

Finally, let us show that the map $\Phi$ given in  $\eqref{CC^1Isomorphism}$ restricts to an isomorphism $A^{p+1}_0(\hat{\C}\setminus \overline D)\to A^p_0(\hat{\C}\setminus \overline D)$. If $f\in A^{p+1}_0(\hat{\C}\setminus \overline D)$ then $g(z)=f(1/z)$ is in $A^{p+1}(D)$ and $g(0)=0$. Then, $\Phi(g)$ is in $A^p(D)$ and is given by
\begin{equation*}\Phi(g)(z)=(f(z^{-1}))'zi+f(\infty)=f'(z^{-1})(-iz^{-1})\end{equation*}
The function $h:D^c\to \C$ given by $h(z)=-\Phi(g)(1/z)$ is in $A^p_0(\hat{\C}\setminus \overline D)$ and
\begin{equation*}h(z)=f'(z)iz\end{equation*}
The assignment $f\mapsto h$ defines an isomorphism $A^{p+1}_0(\hat{\C}\setminus \overline D)\to A^p_0(\hat{\C}\setminus \overline D)$ as it is the composition of isomorphisms $f\mapsto g\mapsto \Phi(g)\mapsto h$. Note that if we restrict $h(z)=f'(z)iz$ on $\T$ we have $h|_{\T}=\Phi(f)$. Therefore, $\Phi:C^{p+1}(\T)\to C^p(\T)$ restricts to an isomorphism $A^{p+1}_0(\hat{\C}\setminus \overline D)\to A^p_0(\hat{\C}\setminus \overline D)$. This completes the proof.
\end{proof}

\begin{thm}If $p<+\infty$, there is an $f\in C^p(\T)$ that can't be written as $f=g+h$ for any $g\in A^p(D)$ and any $h\in A^p_0(\hat{\C}\setminus \overline D)$.\end{thm}

\begin{proof}We induct on $p$. The base case $p=0$ can be found in \cite{Rudin} (compare with Proposition \ref{Functional2}). So assume that $C^p(\T)\neq A^p(D)+A^p_0(\hat{\C}\setminus \overline D)$ while $C^{p+1}(\T)=A^{p+1}(D)+A_0^{p+1}(\hat{\C}\setminus \overline D)$ and take an $f\in C^p(\T)$ not in $A^p(D)+A^p_0(\hat{\C}\setminus \overline D)$. If $\Phi$ is as in Theorem \ref{IsomorphismsDisc}, $\Phi^{-1}(f)\in C^{p+1}(\T)$ hence $\Phi^{-1}(f)=g+h$ for some $g\in A^{p+1}(D)$ and $h\in A^{p+1}_0(\hat{\C}\setminus \overline D)$. But then $f=\Phi(g)+\Phi(h)$ with $\Phi(g)\in A^p(D)$ and $\Phi(h)\in A^p_0(\hat{\C}\setminus \overline D)$, contradicting our assumption on $f$.
\end{proof}
We can actually prove something quite stronger (compare with Proposition \ref{Functional2}):

\begin{thm}\label{A^pNotComplementalC^pCircle}If $p<+\infty$, $A^p(D)$ is not isomorphic to any complemented subspace of $C^p(\T)$.\end{thm}

\begin{proof}The case $p=0$ can be found in \cite{Wojtaszczyk}. We induct on $p$, and assume that $A^p(D)$ is not isomorphic to any complemented subspace of $C^p(\T)$. If there are closed subspaces $K,L$ of $C^{p+1}(\T)$ such that $C^{p+1}(\T)=K\oplus L$ and $K\approx A^{p+1}(D)$ (we use the symbol $\approx$ for isomorphisms), then applying the isomorphism $\Phi$ of Theorem \ref{IsomorphismsDisc} on $C^{p+1}(\T)=K\oplus L$, we obtain that $C^p(\T)=K'\oplus L'$ for $K'\approx K$ and $L'\approx L$. Therefore, $A^p(D)\approx A^{p+1}(D)\approx K\approx K'$ (the first isomorphism is by Theorem \ref{IsomorphismsDisc}) hence $A^p(D)$ is isomorphic to a complemented subspace of $C^p(\T)$, contradicting the induction hypothesis.\end{proof}

As we showed in Theorem $\ref{IsomorphismsDisc}$, all spaces $C^p(\T)$ are isomorphic for $p<+\infty$, and all spaces $A^p(D)$ are isomorphic for $p<+\infty$. This fails if we allow $p=+\infty$:
\begin{thm}\label{CInfinityNoNorm}There is no norm inducing the usual topologies on $C^{\infty}(\T)$ and $A^{\infty}(D)$. Therefore, $C^{\infty}(\T)$ is not isomorphic to $C(\T)=C^0(\T)$ and $A^{\infty}(D)$ is not isomorphic to $A(D)=A^0(D)$.\end{thm}
\begin{proof}We will prove that if $C^{\infty}(\T)$ had a norm $\|\cdot\|$ inducing the same topology as the semi-norms $|\cdot|_l$, then its closed unit ball would be compact, contradicting that it is an infinite dimensional Banach space. Let $\|f_n\|\le 1$ and fix $l$ momentarily. The set 
\begin{equation*}\{f\in C^{\infty}(\T):|f|_l<1\}\end{equation*}
 is an open neighborhood of $0$, so it contains some 
\begin{equation*}\{f\in C^{\infty}(\T):\|f\|<r\}\end{equation*}
Thus, $|rf_n/2|_l<1\implies |f_n|_l<M_l$ uniformly for all $n$, for some $M_l<+\infty$. So it remains to show that $f_n$ has a subsequence converging to some $f\in C^{\infty}$ in all semi-norms $|\cdot|_l$.

The sequence $f_n$ is a uniformly bounded ($\|f_n\|_{\infty}=|f_n|_0<M_0$) and equicontinuous ($\|f'_n\|_{\infty}=|f_n|_1<M_1$) family of continuous functions on the compact set $\T$, so by the Arzel\'a-Ascoli theorem, it has a subsequence $f_{k_n,1}$ converging to some $f\in C(\T)$ uniformly. Similarly, $f_{k_n,1}$ has a subsequence $f'_{k_n,2}$ converging to some $g\in C(\T)$ uniformly; it follows that $f'=g$. Iterating shows that $f\in C^{\infty}(\T)$ and $f_{k_n,n}\to f$ in every semi-norm. The unit ball is thus compact, and we have our contradiction.

This argument can be adapted for $A^{\infty}(D)$ as follows: Let $f_n\in A^{\infty}(D)$ and $|f_n|_l<M_l<+\infty$ uniformly on $n$, and for all $l$. The sequence $f_n\in C(\overline D)$ is uniformly bounded and equicontinuous ($f'_n$ is uniformly bounded and $D$ is convex) so it has a subsequence $f_{k_n,1}$ converging uniformly to a continuous function $f$ on $\bar D$. The function is then holomorphic on $D$ hence  $f\in A(D)$. The same argument gives a subsequence $f_{k_n,2}\in C(\overline D)$ with $f'_{k_n,2}\to g$ uniformly on $\overline D$; obviously $f'=g$ on $D$ so $g$ is the continuous extension of $f'$ on $\T$. Thus, $f_{k_n,2}\to f$ in the topology of $A^1(D)$. Iterating shows that $f\in A^{\infty}(D)$ and $f_{k_n,n}\to f$  in every semi-norm defining the topology of $A^{\infty}(D)$. This completes the proof.
\end{proof}

\section{Jordan Domain}\label{JordanDomainSection}

In this section, $\Omega\subseteq \C$ is a fixed Jordan domain (a simply connected region whose boundary $\partial \Omega$ is a Jordan curve) and $\phi:D\to \Omega$ is a fixed Riemann map.\\
 By the Carath\'eodory-Osgood Theorem, $\phi$ extends to a homeomorphism $\phi:\overline D\to \overline{\Omega}$. The homeomorphism $\gamma:\T\to \partial \Omega$, $\gamma=\phi|_{\T}$, is the parameterization of $\partial \Omega$ that we will be using. 

\begin{defn}We define $C^p(\partial \Omega)$ as the space of functions $f:\partial \Omega\to \C$ such that $f\circ \gamma\in C^p(\T)$, and endow it with semi-norms
\begin{equation*}\Big\|\frac{d^lf}{d\theta^l}\Big\|_{\infty}=\Big\|\frac{d^l(f\circ \gamma)}{d\theta^l}\Big\|_{\infty, \T}\text{ , }0\le l\le p\text{ , }l\in \N\end{equation*}
It is a Fr\'echet space for all $p$ and a Banach space for $p<+\infty$. By definition, $C^p(\partial \Omega)$ is isometrically isomorphic to $C^p(\T)$, via the map $f\mapsto f\circ \gamma$.

$A^p(\Omega)$ is defined as the space of holomorphic functions $f$ on $\Omega$, whose derivatives $f^{(l)}$, $0\le l\le p$, $l\in \N$, extend continuously over $\overline{\Omega}$. Its topology is defined by the semi-norms
\begin{equation*}\|f^{(l)}\|_{\infty, \Omega}=\|f^{(l)}\|_{\infty, \partial\Omega}\text{ , }0\le l\le p\text{ , }l\in \N\end{equation*}
and it is a Fr\'echet space for all $p$ and a Banach space for $p<+\infty$. By definition, $A(\Omega)=A^0(\Omega)$ is isometrically isomorphic to $A(D)$, via the map $f\mapsto f\circ \phi$. 
In Theorem \ref{ApOfOmegaAndDisc} below, we shall give sufficient conditions under which $A^p(\Omega)\approx A^p(D)$ for $0<p<+\infty$. The case $p=+\infty$ is open; see Question \ref{QuestionInfinity} below.

We define $A^p_0(\hat{\C}\setminus \overline{\Omega})$ analogously, with the additional condition that $\lim_{z\to \infty}f(z)=0$.

 Finally, we define
\begin{equation*}\overline{A_0^p(\Omega)}=\{\bar f: f\in A^p(\Omega)\text{ , }f(\phi(0))=0\}\end{equation*}
is topologized by semi-norms
\begin{equation*}\|\big(\overline f\big)^{(l)}\|_{\infty, \Omega}=\|\big(\overline f\big)^{(l)}\|_{\infty, \partial\Omega}\text{ , }0\le l\le p\text{ , }l\in \N\end{equation*}
and the conjugation map $ \overline{A_0^p(\Omega)}\to A_0^p(\Omega)$ becomes an isometric isomorphism.\end{defn}

The decomposition in Theorem \ref{CInfinityDecomposition} is the trickiest to generalize in this setting, so we postpone this until the end of this section, and instead proceed with extending the decomposition in Theorem \ref{CInfinityConjugateDisc}:

\begin{thm}\label{ConjugateJordanDomain}If the Riemann map $\phi:D\to \Omega$ is in $A^{\infty}(D)$ and $\phi'(z)\neq 0$ for all $z\in \T$, then every $f\in C^{\infty}(\partial \Omega)$ has a unique decomposition as $f=g+\bar h$ with $g,h\in A^{\infty}(\Omega)$ and $h(\phi(0))=0$.\end{thm}

\begin{proof}If $f\in C^{\infty}(\partial \Omega)$ then $f\circ \gamma\in C^{\infty}(\T)$ can be decomposed as $f\circ \gamma=g+\bar h$, for $g,h\in A^{\infty}(D)$ and $h(0)=0$, by Theorem \ref{CInfinityConjugateDisc}. So 
\begin{equation*}f=g\circ \gamma^{-1}+\bar h\circ \gamma^{-1}=g\circ \gamma^{-1}+\overline{h\circ \gamma^{-1}}\end{equation*}
 and $g\circ \gamma^{-1},h\circ \gamma^{-1}\in A^{\infty}(\Omega)$ (their analytic extensions over $\Omega$ are $g\circ \phi^{-1}$ and $h\circ \phi^{-1}$; we have $\phi^{-1}\in A^{\infty}(\Omega)$ because $\phi\in A^{\infty}(\Omega)$ and $\phi'\neq 0$) and $(h\circ \phi^{-1})(\phi(0))=h(0)=0$.

 The uniqueness of the decomposition $f=g+h$ follows from the fact if $f=0$ then $g=\bar h$ for holomorphic $g,h$ on a domain, hence both $g$ and $h$ are constant.\end{proof}

We now examine when $A^{\infty}(\Omega)$ embeds in $C^{\infty}(\partial \Omega)$ in the canonical way, and more generally when $A^p(\Omega)$ embeds in $C^p(\partial \Omega)$, $0\le p\le +\infty$:

\begin{thm}Let $0\le p\le +\infty$ be fixed. The following are equivalent
\begin{itemize}\item[1.] The map $A^p(\Omega)\to C^p(\partial \Omega)$ given by restriction to the boundary is well defined; that is, $f|_{\partial \Omega}\in C^p(\partial \Omega)$ for every $f\in A^p(\Omega)$.
\item[2.] $\gamma\in C^p(\T)$.
\item[3.] $\phi\in A^p(D)$.\end{itemize}\end{thm}

\begin{proof}This is trivial for $p=0$ (by continuity), so assume $p\ge 1$. The equivalence of items 2 and 3 is part of Proposition \ref{A^pC^pDerivatives} (see \cite{Vlasis}).

If the restriction map $A^p(\Omega)\to C^p(\partial \Omega)$  is well defined, then $\id_{\Omega}\in A^p(\Omega)$ would restrict to $\id_{\partial \Omega}\in C^p(\partial \Omega)$, namely $\id_{\partial \Omega}\circ \gamma\in C^p(\T)\iff \gamma\in C^p(\T)$.

For the converse, first take $p=1$ and let $f\in A^1(\Omega)$ and $\widetilde f,\widetilde{f'}$ be the extensions of $f,f'$ over the boundary. We want to show that $\widetilde f\circ \gamma$ is $C^1$ smooth.

For $0\le r<1$ set $\gamma_r(e^{i\theta})=\phi(re^{i\theta})$. For convenience, we denote $u'=\frac{du}{d\theta}$ for $u\in C^1(\T)$. By the uniform continuity of $\phi$ we have that $\gamma_r\to \gamma$ uniformly. In addition, $\gamma'_r\to \gamma'$ uniformly, since for $r<1$ we have
\begin{align*}|\gamma'_r(e^{i\theta})-\gamma'(e^{i\theta})|=&|\phi'(re^{i\theta})r-\phi'(e^{i\theta})|\le \\
&|\phi'(re^{i\theta})||r-1|+|\phi'(re^{i\theta})-\phi'(e^{i\theta})|\le \\
&\|\phi'\|_{\infty}(1-r)+|\phi'(re^{i\theta})-\phi'(e^{i\theta})|\end{align*}
which is arbitrarily small for $r$ sufficiently close to $1$, by the uniform continuity of $\phi'$. So $f\circ \gamma_r=\widetilde f\circ \gamma_r\to \widetilde f\circ \gamma$ uniformly, and $(f\circ \gamma_r)'=(f'\circ \gamma_r)\gamma'_r\to (\widetilde{f'}\circ \gamma)\gamma'$ uniformly. By a theorem in real analysis, we conclude that $\widetilde f\circ \gamma$ is differentiable and
\begin{equation}\label{DerivativeWithRespectTot}\frac{d\widetilde f}{d\theta}=(\widetilde{f'}\circ \gamma)\gamma'\end{equation}
An induction on $p$ completes the proof.
\end{proof}

So under the assumption $\gamma\in C^p(\T)$, we have $A^p(\Omega)$ as a subset of $C^p(\T)$. To have it embedded in $C^p(\T)$ as a closed subspace, we additionally need the usual topology on $A^p(\Omega)$ given by the semi-norms
\begin{equation*}\|f^{(l)}\|_{\infty,\Omega}=\|f^{(l)}\|_{\infty,\partial \Omega}\text{ , }0\le l\le p\text{ , }l\in \N\end{equation*}
to agree with the relative topology induced by $C^p(\T)$, i.e., given by the semi-norms
\begin{equation*}\Big\|\frac{d^l(f\circ \gamma)}{d\theta^l}\Big\|_{\infty, \T}\text{ , }0\le l\le p\text{ , }l\in \N\end{equation*}
A sufficient condition is that $\gamma'\neq 0$ (for $p\ge 1$ of course). Indeed, by the equations
\begin{equation*}\frac{df}{d\theta}=\frac{df}{dz}\gamma'\iff \frac{df}{dz}=\frac{df}{d\theta}(\gamma')^{-1} \end{equation*}
we can prove that the two semi-norms for $l=1$ are equivalent (they are equal for $l=0$). This can be done for all $l\le p$, $l\in \N$, by induction, see for instance \cite{Bolkas} .
\begin{thm}\label{ApClosedInCp}Let $p\ge 1$. If $\gamma\in C^p(\T)$ and $\gamma'\neq 0$ then $A^p(\Omega)$ is a closed subspace of $C^p(\partial \Omega)$. For $p=0$, we always have $A(\Omega)=A^0(\Omega)$ as a closed subspace of $C(\partial \Omega)=C^0(\partial \Omega)$.\end{thm}

\begin{proof}The image of $A^p(\Omega)$ in $C^p(\partial \Omega)$ under the restriction map, $f\mapsto f|_{\partial \Omega}$, is isomorphic to the complete space $A^p(\Omega)$ by the preceding discussion. Therefore, $A^p(\Omega)$ is a closed subspace of $C^p(\partial \Omega)$\end{proof}
The analogous statement is true for $\overline{A^p(\Omega)}$. Therefore, by Theorem \ref{ConjugateJordanDomain} and Proposition \ref{Functional2}, we have that

\begin{cor}If $\gamma\in C^{\infty}(\T)$ and $\gamma'\neq 0$ then $C^{\infty}(\partial \Omega)=A^{\infty}(\Omega)\oplus \overline{A_0^{\infty}(\Omega)}$.\end{cor}

We shall prove that no such decomposition can occur for $p<+\infty$, under some mild conditions on $\Omega$. Precisely, we define the family of Jordan domains $\mathcal F$ consisting of all $\Omega$ with one of the following properties:
\begin{itemize}\item[1.] $\Omega$ has rectifiable boundary.
\item[2.] $\Omega$ is star-like.
\item[3.] $\Omega$ is bounded by the graph of a continuous function $s:[0,1]\to \R$ and the horizontal axis positioned at height $c<\min s$:
\begin{equation*}\Omega=\{(x,y)\in \C: 0< x< 1\text{ , }c<y<s(x)\}\end{equation*}
\item[4.] For any $\epsilon>0$ there is a $\delta>0$ so that whenever $z,w\in \Omega$, $|z-w|<\delta$, there is a curve in $\Omega$ connecting them with length less than $\epsilon$.
\item[5.] There is a $C>0$ such that for any two points $z,w\in \Omega$ there is a curve in $\Omega$ connecting them with length less than $C|z-w|$. This is the so called  ''interior chord arc condition'' (\cite{Chord Arc}).
\item[6.] The Riemann map $\phi:D\to \Omega$ is Lipschitz continuous. This is equivalent to Lipschitz continuity on $\overline D$, and also equivalent to $\phi'$ being bounded on $D$ (because $D$ is convex).
\end{itemize}

We remark that each one of conditions 5 and 6 imply condition 4 (if $\phi$ is Lipschitz then we can use curves $\phi([\phi^{-1}(z),\phi^{-1}(w)])$ to connect $z,w$), while condition 6 clearly implies condition 1.

\begin{que}Do conditions 4 or 5 imply condition 1 ?\end{que}

\begin{thm}If $\Omega\in \mathcal F$ then every function in $A(\Omega)$ has an antiderivative in $A^1(\Omega)$; in other words, the integration operator on $\Omega$ maps $A(\Omega)$ to itself.\end{thm}

\begin{proof}We refer the reader to \cite{Nestoridis} for the cases of conditions 1,2 and 3. We shall only treat condition 4, which is weaker than conditions 5 and 6. 

Assume that for any $\epsilon>0$ there is a $\delta>0$ such that whenever $z,w\in \Omega$, $|z-w|<\delta$, then there is a rectifiable curve $\delta_{z,w}$ in $\Omega$ connecting $z,w$ with length $\ell(\delta_{z,w})$ less than $\epsilon$. For arbitrary curves $\gamma_z$ in $\Omega$ starting from a fixed $z_0\in \Omega$ to $z$, the function $F:D\to \C$ given by $F(z)=\int_{\gamma_z}f(\zeta)\, d\zeta$ is an antiderivative of $f$ in $\Omega$, and is independent of the choice of curves $\gamma_z$.  For $z,w\in \Omega$ and $|z-w|<\delta$ we have

\begin{equation*}|F(z)-F(w)|=\Big|\int_{\delta_{z,w}}f(\zeta)d\zeta\Big|\le \|f\|_{\infty}\ell(\delta_{z,w})<\|f\|_{\infty}\epsilon\end{equation*}
which can become arbitrarily small as $\|f\|_{\infty}<+\infty$. Therefore, $F$ is uniformly continuous on $\Omega$, hence extends continuously over $\overline{\Omega}$. Because $F'=f\in A(\Omega)$, we conclude that $F\in A^1(\Omega)$.\end{proof}

\begin{rem}This proof also yields that if $\Omega$ has property 4, then every function in $H^{\infty}(\Omega)$ (a function holomorphic and bounded on $\Omega$) has an antiderivative in $A(\Omega)$. The integration operator maps $H^{\infty}(\Omega)$ to itself if and only if there is an $M\in (0,+\infty)$ so that any two points in $\Omega$ can be joined by a curve in $\Omega$ with length at most $M$ (\cite{HInfinity}).\end{rem}

\begin{thm}\label{ApOfOmegaAndDisc}If the integration operator on $\Omega$ maps $A(\Omega)$ to itself and $p<+\infty$, then $A^p(\Omega)\approx A^{p+1}(\Omega)$ and $A^p(\Omega)\approx A^p(D)$. In particular, the conclusion holds for $\Omega\in \mathcal F$.\end{thm}

\begin{proof}The first statement follows exactly as in the case of the disc: After a translation, we may assume $0\in \Omega^{\circ}$. The map $\Phi:A^{p+1}(\Omega)\to A^p(\Omega)$ given by
\begin{equation*}\Phi(f)(z)=f'(z)zi+f(0)\end{equation*}
is an isomorphism, the proof of which is similar to the second part of the proof of Theorem \ref{IsomorphismsDisc}. In addition, $A^p(\Omega)\approx A(\Omega)\approx A(D)\approx A^p(D)$ where the second isomorphism is given by $f\mapsto f\circ \phi$.
\end{proof}

As $C^{\infty}(\T)\approx C^{\infty}(\partial \Omega)$, $C^{\infty}(\partial \Omega)$ has no norm inducing its usual topology by Theorem \ref{CInfinityNoNorm}.
 
\begin{thm} $A^{\infty}(\Omega)$ has no norm inducing its usual topology, if $\Omega$ has property 4 (or properties 5,6). \end{thm}

\begin{proof}This is a straightforward adaptation of the proof of Theorem \ref{CInfinityNoNorm}. The only difference is in showing that uniform boundedness of $f'_n$ on $\overline{\Omega}$ implies equicontinuity of $f_n$ on $\overline{\Omega}$. Here is where the condition 4 comes into play: For arbitrary $\epsilon>0$ take $\delta>0$ so that any two $z,w\in \Omega$ less than $\delta$ apart can be connected by a curve $\delta_{z,w}$ in $\Omega$ with length less than $\epsilon/M$, $M$ being the uniform bound on $f'_n$. Then,
\begin{equation*}|f_n(z)-f_n(w)|=\Big|\int_{\delta_{z,w}}f_n'(\zeta)d\zeta\Big|\le M\ell(\delta_{z,w})<\epsilon\end{equation*}
where by $\ell(\delta_{z,w})$ we denote the length of the curve $\delta_{z,w}$. So we have uniform equicontinuity on $\Omega$, which in turn implies equicontinuity on $\overline{\Omega}$. The rest of the proof is similar to the proof of Theorem \ref{CInfinityNoNorm}.
\end{proof}

\begin{que}\label{QuestionInfinity}Under what assumptions on $\Omega$ are the spaces $A^{\infty}(D)$ and $A^{\infty}(\Omega)$ isomorphic?\end{que}

We now generalize Theorem \ref{A^pNotComplementalC^pCircle}

\begin{thm}\label{ApNotComplemental}If either $p=0$ or $1\le p<+\infty$ and the integration operator maps $A(\Omega)$ to itself, then $A^p(\Omega)$ is not isomorphic to any complemented subspace of $C^p(\partial \Omega)$. In particular, this holds for $\Omega\in \mathcal F$ and $p<+\infty$.\end{thm}

\begin{proof}First take $p=0$ and assume there are $K,L$ so that $C(\partial \Omega)=K\oplus L$ and $A(\Omega)\approx K$. We apply the isomorphism $C(\partial \Omega)\to C(\T)$ (given by $f\mapsto f\circ \gamma$) on $C(\partial \Omega)=K\oplus L$  to obtain that $C(\T)=K'\oplus L'$ for $K',L'$ isomorphic to $K.L$ respectively. But then $K'\approx K\approx A(\Omega)\approx A(D)$ and $K'$ is complemented in $C(\T)$, contradicting Theorem \ref{A^pNotComplementalC^pCircle} (\cite{Wojtaszczyk}).

Now take $\Omega\in \mathcal F$, $p<+\infty$, and assume that $C^p(\partial \Omega)=K\oplus L$ for some $K,L$ with $A^p(\Omega)\approx K$. As before, we have $C^p(\T)\approx K'\oplus L'$ for $K'\approx K$. The space $K$ is isomorphic to $A^p(\Omega)$, which by Theorem \ref{ApOfOmegaAndDisc} is isomorphic to $A^p(D)$; consequently, $A^p(D)$ is isomorphic to the complemented subspace $K'$ of $C^p(\T)$, contradicting Theorem \ref{A^pNotComplementalC^pCircle}.\end{proof}

\begin{prop}\label{TrivialIntersectionApOmega}If $\partial \Omega$ has continuous analytic capacity $0$ then $A^p(\Omega)\cap A^p_0(\hat{\C}\setminus \overline{\Omega})$ is trivial for all $p$.\end{prop}

\begin{proof}A function $f$ in the intersection of the two spaces would have to be continuous on $\C$ and holomorphic on $\C\setminus \partial \Omega$. The boundary $\partial \Omega$ having zero continuous analytic capacity translates to the fact that such $f$ is entire (\cite{Bolkas}, \cite{Garnett}). Since $\lim_{z\to \infty}f(z)=0$ we have by Liouville's Theorem that $f$ is identically zero.\end{proof}

The complement of $\overline{\Omega}$ in the Riemann sphere is simply connected, so there is a Riemann map $\psi: {\overline D}^c\to {\overline {\Omega}}^c$. By the Carath\'eodory-Osgood Theorem, the map $\psi$ extends over the boundaries $\partial D,\partial \Omega$, as a homeomorphism $\delta:\T\to \partial \Omega$. The composition $\delta\circ \gamma^{-1}$ is a homeomorphism of $\partial \Omega$, and $\gamma^{-1}\circ \delta$ is a homeomorphism of $\T$. As the only injective entire maps $\C\to \C$ are linear, $\gamma\neq \delta$ unless $\Omega$ is a disc in the plane; the homeomorphism  $\gamma^{-1}\circ \delta:\T\to \T$ is usually not the identity mapping. Any homeomorphism of $\T$ obtained this way (for arbitrary Jordan domain $\Omega$) is called a welding (\cite{Elliptic}, \cite{Welding}).

\begin{prop}\label{ApOfComplementClosedSubspace}If $\delta\in C^p(\T)$ and $\delta'\neq 0$ then $A^p_0(\hat{\C}\setminus \overline{\Omega})$ is a closed subspace of $C^p(\partial \Omega)$, embedded via the restriction map.\end{prop}

\begin{proof}Analogous to the proof of Theorem \ref{ApClosedInCp}.\end{proof}

\begin{thm}If either one of the following items are true
\begin{itemize}\item $p=0$ and $\partial \Omega$ has continuous analytic capacity $0$
\item $1\le p<+\infty$ and  $\gamma,\delta\in C^p(\T)$, $\gamma',\delta'\neq 0$ \end{itemize}
then there is a function in $C^p(\partial \Omega)$ that can't be decomposed as the sum of a function in $A^p(\Omega)$ and another in $A^p_0(\hat{\C}\setminus \overline{\Omega})$.\end{thm}

\begin{proof}Under our conditions, $A^p(\Omega)$ and $A^p_0(\hat{\C}\setminus \overline{\Omega})$ are closed subspaces of $C^p(\partial \Omega)$ (Propositions \ref{ApClosedInCp} and \ref{ApOfComplementClosedSubspace}) with trivial intersection (Proposition \ref{TrivialIntersectionApOmega}). Therefore, by Proposition \ref{Functional2}, if we assume that $C^p(\partial \Omega)=A^p(\Omega)+A^p_0(\hat{\C}\setminus \overline{\Omega})$ then $A^p(\Omega)$ is complemented in $C^p(\partial \Omega)$, contradicting Theorem \ref{ApNotComplemental}. Thus, $C^p(\partial \Omega)\neq A^p(\Omega)+A^p_0(\hat{\C}\setminus \overline{\Omega})$.\end{proof}

We shall now describe some other splittings of $C^{\infty}(\T)$ in terms of $A^{\infty}(\Omega)$, $A_0^{\infty}(\hat{\C}\setminus \overline{\Omega})$ and the welding $\delta\circ \gamma^{-1}$ (and its inverse). But before we do that, let us fix some notation:

\begin{nota} If $A$ is a set of functions $X\to Y$ and $g:Z\to X$ then we denote $A\circ g=\{f\circ g: f\in A\}$.\end{nota}

\begin{prop}Assume $\gamma,\delta\in C^{\infty}(\T)$ and $\gamma',\delta'\neq 0$. Then

\begin{itemize}\item $C^{\infty}(\partial \Omega)=A^{\infty}(\Omega)\oplus [A_0^{\infty}(\hat{\C}\setminus \overline{\Omega})\circ \delta\circ \gamma^{-1}]$
\item $C^{\infty}(\partial \Omega)=[A^{\infty}(\Omega)\circ \gamma\circ \delta^{-1}]\oplus A_0^{\infty}(\hat{\C}\setminus \overline{\Omega})$\end{itemize}\end{prop}

\begin{proof}Let $P:C^{\infty}(\T)\to A^{\infty}(D)$ be the canonical projection, $P(\sum_{n=-\infty}^{+\infty}a_nz^n)=\sum_{n=0}^{\infty}a_nz^n$. Then $Q:C^{\infty}(\partial \Omega)\to A^{\infty}(\Omega)$ defined by $Q(f)=P(f\circ \gamma)\circ \gamma^{-1}$ is a projection (it fixes $A^{\infty}(\Omega)$ precisely because $P$ does). Therefore, $C^{\infty}(\partial \Omega)=A^{\infty}(\Omega)\oplus KerQ$. We determine the kernel:
\begin{align*}Q(f)=0\iff &P(f\circ \gamma)=0\iff \\
&f\circ \gamma\in A_0^{\infty}(\hat{\C}\setminus \overline D)\iff f\circ \gamma\circ \delta^{-1}\in A_0^{\infty}(\hat{\C}\setminus \overline{\Omega})\iff\\
& f\in A_0^{\infty}(\hat{\C}\setminus \overline{\Omega})\circ \delta\circ \gamma^{-1} \end{align*}
The other item follows similarly.\end{proof}

We shall now examine how the decomposition in Theorem \ref{CInfinityDecomposition} extends to our Jordan domain $\Omega$.

\begin{thm}For $\gamma,\delta\in C^{\infty}(\T)$ with $\gamma',\delta'\neq 0$ the following are equivalent
\begin{itemize}
\item[1.] Every function in $C^{\infty}(\partial \Omega)$ has a unique decomposition as the sum of a function in $A^{\infty}(\Omega)$ and a function in $A _0^{\infty}(\hat{\C}\setminus \overline{\Omega})$\smallbreak
\item[2.] $C^{\infty}(\partial \Omega)=A^{\infty}(\Omega)\oplus A_0^{\infty}(\hat{\C}\setminus \overline{\Omega})$\smallbreak
\item[3.] $C^{\infty}(\T)=A^{\infty}(D)\oplus  [A^{\infty}_0(\hat{\C}\setminus \overline D)\circ \delta^{-1}\circ \gamma]$\smallbreak
\item[4.]  $C^{\infty}(\T)=[A^{\infty}(D)\circ \gamma^{-1}\circ \delta]\oplus  A^{\infty}_0(\hat{\C}\setminus \overline D)$
\end{itemize}
\end{thm}

\begin{proof}The first and second items are equivalent by Proposition \ref{Functional2}. Let us now show the equivalence of the second and third items.

 Let $P:C^{\infty}(\partial \Omega)\to A^{\infty}(\Omega)$ be the projection corresponding to the splitting $C^{\infty}(\partial \Omega)=A^{\infty}(\Omega)\oplus A_0^{\infty}(\hat{\C}\setminus \overline{\Omega})$. Then $Q:C^{\infty}(\T)\to A^{\infty}(D)$, $Q(u)=P(f\circ \gamma^{-1})\circ \gamma$, is a projection as usual, hence $C^{\infty}(\T)=A^{\infty}(D)\oplus Ker Q$. We determine the kernel
\begin{align*}Q(f)=0\iff &P(f\circ \gamma^{-1})=0\iff \\
&f\circ \gamma^{-1}\in A_0^{\infty}(\hat{\C}\setminus \overline{\Omega})\iff f\circ \gamma^{-1}\circ \delta\in A_0^{\infty}(\hat{\C}\setminus \overline D)\iff\\
&f\in A_0^{\infty}(\hat{\C}\setminus \overline D)\circ \delta^{-1}\circ \gamma
\end{align*}
The equivalence of the second and fourth items is similar.
\end{proof}

\begin{que}Is there a Jordan domain $\Omega$ and a function $f\in C^{\infty}(\partial \Omega)$ that can't be decomposed as $f=g+h$ for any $g\in A^{\infty}(\Omega)$ and $h\in A_0^{\infty}(\hat{\C}\setminus \overline{\Omega})$?\end{que}

We suspect the answer is affirmative for the following reason: The map $w(e^{i\theta})=e^{-i\theta}$ is a welding by the conformal welding theorem and certainly, $A^{\infty}(D)\circ w+ A^{\infty}_0(\hat{\C}\setminus \overline D)$ is not a direct sum. But we don't yet know if the functions $\gamma,\delta$ realizing the welding are $C^{\infty}$ diffeomorphisms of $\T$. 

The conformal welding theorem states in particular that every quasi-symmetry of the circle is a welding (\cite{Elliptic},\cite{Welding Theorem}). An injection $f:A\to \C$, $A\subseteq \C$, is quasi-symmetric, if there is an increasing homeomorphism $\eta:[0,+\infty)\to [0,+\infty)$ with
\begin{equation*} \frac{|f(x)-f(y)|}{|f(x)-f(z)|} \leq \eta\left(\frac{|x-y|}{|x-z|}\right)\end{equation*}
for all triples $x,y,z\in A$, $x\neq z$. Clearly, diffeomorphisms $\T\to \T$ are quasi-symmetric.

\section{Internally Tangent Circles}\label{InternallTangentSection}

Take a disc $D'$ in the interior of $D$, with $\partial D'$ tangent at $1$ to $\T=\partial D$, and let $\Omega=D-\overline{D'}$. We denote by $\T$ the boundary of $D$ (as usual) and the boundary of $D'$ by $\gamma$. We examine whether or not every function $f\in A^p(\Omega)$ has a decomposition as $f=g+h$ for $g\in A^p(D)$ and $h\in A^p_0(\hat {\C}\setminus \overline{D'})$. The case $p=+\infty$ is easily dealt with:

\begin{thm}Every function in $A^{\infty}(\Omega)$ can be written uniquely as the sum of a function in $A^{\infty}(D)$ and of a function in $A^{\infty}_0(\hat \C\setminus \overline{D'})$.\end{thm}

\begin{proof}If $f\in A^{\infty}(\Omega)$ then clearly, $f\in C^{\infty}(\T-\{1\})$. The derivatives of $f$ extend continuously over $\T$ hence $f\in C^{\infty}(\T)$; we similarly have that $f\in C^{\infty}(\partial D')$.  By Theorem \ref{CInfinityDecomposition}, there are $g\in A^{\infty}(D)$ and $h\in A^{\infty}_0(\hat \C\setminus \overline D)$ so that $f=g+h$ on the unit circle. We extend $h$ on $D'^c$ by setting $h=f-g$ over $D-D'$; the derivatives of $h$ are also continuously extended this way. It is easy to see that the extended $h$ is holomorphic over $\overline{D'}^c$, so $h\in A^{\infty}_0(\hat \C\setminus \overline{D'})$ as desired. This proves the existence of the decomposition. For the uniqueness of this decomposition, we observe that Liouville's Theorem implies that $A^{\infty}(D)\cap A^{\infty}_0(\hat \C\setminus \overline{D'})=0$.\end{proof}

Combining this with Proposition \ref{Functional2} we obtain:

\begin{cor}$A^{\infty}(\Omega)=A^{\infty}(D)\oplus A^{\infty}_0(\hat \C\setminus \overline{D'})$.\end{cor}

\begin{que}For $p<+\infty$, is there an $F\in A^p(\Omega)$ that can't be written as the sum of a function in $A^p(D)$ and another in $A^p_0(\hat{\C}\setminus \overline {D'})$? \end{que}

 We suspect that the decomposition does not hold (just as in the case for the circle/Jordan curve) and will now collect some sufficient conditions for a function $F\in A^p(\Omega)$ to not be in $A^p(D)+A^p_0(\hat \C\setminus \overline{D'})$. \\
The Cauchy transform is defined as
\begin{equation*}C_{\T}(F)=\frac1{2\pi i}\int_{\T}\frac{f(\zeta)}{\zeta-z}d\zeta\end{equation*}
Assume $F=f+g$, $f\in A^p(D)$ and $g\in A^p_0(\hat \C\setminus \overline{D'})$. Then $C_{\T}(F)=C_{\T}(f)+C_{\T}(g)$. We calculate
\begin{equation*}C_{\T}(F)=\begin{cases}f(z)&\textup{if, }z\in D\\
-g(z)&\textup{if, }z\notin D\end{cases}\end{equation*}
If $\gamma$ is the boundary of $\partial D'$, we can similarly prove that
\begin{equation*}C_{\gamma}(F)=\begin{cases}f(z)&\textup{if, }z\in D'\\
-g(z)&\textup{if, }z\notin D'\end{cases}\end{equation*}
We have
\begin{equation*}\lim_{z\to 1,z\in D}C_{\T}(F)=\lim_{z\to 1,z\in \Omega}C_{\T}(F)=\lim_{z\to 1,z\in D'}C_{\gamma}(F)\in \C\end{equation*}
and
\begin{equation*}\lim_{z\to 1,z\notin D}C_{\T}(F)=\lim_{z\to 1,z\notin D}C_{\gamma}(F)=\lim_{z\to 1,z\in \Omega}C_{\gamma}(F)\in \C\end{equation*}

\begin{que}\label{FinalQuestion}Let $p<+\infty$. Does there exist $F\in A^p(\Omega)$ so that either one of the aforementioned limits doesn't exist, or a pair of limits exists but the limits are not equal? In particular, is there a function $F\in A(D)$ such that $\lim_{z\to 1,|z|<1}C_{\T}(F)$ does not exist (in $\C$) ?\end{que}

\noindent \textbf{Acknowledgment}: We would like to thank V. Nestoridis for bringing these problems to our attention, and for his always useful suggestions and help. We would also like to thank S. Mercourakis for mentioning reference \cite{Wojtaszczyk} to us.


\begin{thebibliography}{}

\bibitem{Chord Arc}   J. M. Anderson, J. Becker, and J. Gevirtz, {\em First-order univalence criteria, interior chord-arc conditions and quasidisks}, Michigan Math. J. Volume 56, Issue 3 (2008), 623-636.
\smallbreak
\bibitem{Elliptic} K. Astala, T. Iwaniec, and G. Martin, {\em Elliptic Partial Differential Equations and Quasiconformal Mappings in the Plane},  Princeton University Press (2009).
\smallbreak
\bibitem{Welding}  C. J. Bishop, {\em Conformal welding and Koebe's theorem}, Annals of Mathematics, 166 (2007), 613-656
\smallbreak
\bibitem{Bolkas}  E. Bolkas, V. Nestoridis, C. Panagiotis, {\em Non extendability from any side of the domain of definition as a generic property of smooth or simply continuous functions on an analytic curve}, arXiv:1511.08584.
\smallbreak
\bibitem{Costakis}   G. Costakis, V. Nestoridis, and I. Papadoperakis, {\em Universal Laurent series}, Proc. Edinb. Math. Soc. (2) 48 (2005), no. 3, 571-583.
\smallbreak
\bibitem{Garnett}  J. Garnett, {\em Analytic Capacity and Measure}, Springer Verlag (1972)
\smallbreak
\bibitem{Vlasis}  V. Mastrantonis, {\em Relations of the $A^p$ and $C^p$ spaces}, in preparation.
\smallbreak
\bibitem{Nestoridis}   V. Nestoridis and I. Zadik, {\em Pade Approximants, density of rational functions in $A^{\infty}(\Omega)$ and smoothness of the integration operator}, arXiv:1212.4394.
\smallbreak
\bibitem{Rudin}  W. Rudin, {\em Functional Analysis}, McGraw-Hill  (1973)
\smallbreak
\bibitem{Welding Theorem}   E. Schippers and W. Staubach, {\em A sympletic function analysis proof of the conformal welding theorem}, Proc. Amer. Math. Soc. 143 (2015), 265-278
\smallbreak
\bibitem{HInfinity}   W. Smith, D. M. Stolyarov, A. Volberg, {\em On Bloch approximation and the boundedness of integration operator on $H^{\infty}$}, arXiv:1604.05433.
\smallbreak
\bibitem{Wojtaszczyk}  P. Wojtaszczyk, {\em Banach Spaces For Analysts}, Cambridge University Press (1991)

\end{thebibliography}
\end{document}